\documentclass[11pt,reqno]{amsart}
\usepackage{amsfonts}
\usepackage{amssymb, amsmath, latexsym}
\usepackage{lipsum}
\usepackage{mathtools}

\usepackage{array}
\usepackage{cite}

\setbox0=\hbox{$+$}
\newdimen\plusheight
\plusheight=\ht0
\def\+{\;\lower\plusheight\hbox{$+$}\;}

\setbox0=\hbox{$-$}
\newdimen\minusheight
\minusheight=\ht0
\def\-{\;\lower\minusheight\hbox{$-$}\;}

\setbox0=\hbox{$\cdots$}
\newdimen\cdotsheight
\cdotsheight=\plusheight
\def\cds{\lower\cdotsheight\hbox{$\cdots$}}

\usepackage{mleftright}
\numberwithin{equation}{section}
\theoremstyle{plain}
\newtheorem{theorem}{Theorem}[section]

\newtheorem{lemma}[theorem]{Lemma}
\newtheorem{conjecture}[theorem]{Conjecture}
\newtheorem{corollary}[theorem]{Corollary}
\newtheorem{proposition}[theorem]{Proposition}
\newtheorem{remark}[theorem]{Remark}

\newtheorem{question}[theorem]{Question}
\newtheorem{definition}[theorem]{Definition}
\newtheorem{example}[theorem]{Example}
\newtheorem{observation}[theorem]{Observation}

\makeatletter
\def\@bignumber#1#2{%
  \ifx#2\end
    #1\let\next\@gobble
  \else
    #1\hspace{0pt plus 1pt}\let\next\@bignumber
  \fi
  \next#2}
\newcommand{\bignumber}[1]{\@bignumber#1\end}
\makeatother

\begin{document}
\allowdisplaybreaks
\title[Further aspects of $\mathcal{I}^{\mathcal{K}}$-convergence in Topological Spaces] {Further aspects of $\mathcal{I}^{\mathcal{K}}$-convergence in\\ Topological Spaces}

\author{Ankur Sharmah}

\address{Department of Mathematical Sciences, Tezpur University, Napam 784028, Assam, India}
\email{ankurs@tezu.ernet.in}

\author{Debajit Hazarika}
\address{Department of Mathematical Sciences, Tezpur University, Napam 784028, Assam, India}
\email{debajit@tezu.ernet.in}



\maketitle
\begin{abstract}
		In this paper, we obtain some results on the relationships between different ideal \linebreak convergence modes namely, $\mathcal{I}^\mathcal{K}$, $\mathcal{I}^{\mathcal{K}^*}$, $\mathcal{I}$, $\mathcal{K}$, $\mathcal{I} \cup \mathcal{K}$ and $(\mathcal{I} \cup \mathcal{K})^*$. We introduce a topological space namely $\mathcal{I}^\mathcal{K}$-sequential space and show that the class of  $\mathcal{I}^\mathcal{K}$-sequential spaces contain the sequential spaces. Further $\mathcal{I}^\mathcal{K}$-notions of cluster points and limit points of a function are also introduced here. For a given sequence in a topological space $X$, we characterize the set of $\mathcal{I}^\mathcal{K}$-cluster points of the sequence as closed subsets of $X$.
\end{abstract}
\section{ Introduction}\label{intro}

For basic general topological terminologies and results we refer to \cite{HNV04}.
The ideal convergence of a sequence of real numbers was introduced by Kostyrko et al. \cite{KSW01}, as a natural generalization of existing convergence notions such as usual convergence \cite{HNV04}, statistical convergence \cite{F51}. It was further introduced in arbitrary topological spaces accordingly for sequences \cite{LD05} and nets \cite{LD07} by Das et al.  The main goal of this article is to study $\mathcal{I}^{\mathcal{K}}$-convergence which arose as a generalization of a type of ideal convergence. In this continuation we begin with a prior mentioning of ideals and ideal convergence in topological spaces.

\vspace{1mm}
An ideal $\mathcal{I}$ on a arbitrary set $\mathcal{S}$ is a family $\mathcal{I} \subset 2^{S}$ (the power set of S) that is closed under finite unions and taking subsets. $Fin$ and $\mathcal{I}_0$ are two basic ideals on $\omega$, the set of all natural numbers, defined as $Fin$:= collection of all finite subsets of $\omega$ and $\mathcal{I}_0$:= subsets of $\omega$ with density $0$, we say $A (\subset \omega )\in \mathcal{I}_0$ if and only if $ \lim	sup_{n \rightarrow \infty}  \frac{|A \cap \{1, 2, ..., n\}|}{n}= 0$. For an ideal $\mathcal{I}$ in  $P(\omega)$, we have two additional subsets of $P(\omega)$ namely $\mathcal{I}^{\star}$ and $\mathcal{I}^+$, where $\mathcal{I}^{\star} := \{A \subset \omega : A^c \in \mathcal{I}\}$, the filter dual of $\mathcal{I}$ and $\mathcal{I}^+$:= collection of all subsets not in $\mathcal{I}$. Clearly, $\mathcal{I}^{\star} \subseteq \mathcal{I}^+$.
 A sequence $x = \{x_n\}_{n \in \omega}$ is said to be $\mathcal{I}$-convergent \cite{LD05} to $\xi$, denoted by $x_n \rightarrow _\mathcal{I} \xi$, if $\{n : x_n \notin {U}\} \in \mathcal{I}$, $\forall$ neighborhood ${U}$ of $\xi$. A sequence $x = \{x_n\}_{n \in \omega}$ of elements of $X$ is said to be $\mathcal{I^{\star}}$-convergent to $\xi$ if there exists a set $M := \{m_1 < m_2 <... < m_k <... \} \in \mathcal{I^{\star}}$ such that $\lim_{k \to \infty}x_{m_k} = \xi$. Lahiri and Das \cite{LD05} found an equivalence between $\mathcal{I}$ and $\mathcal{I^{\star}}$-convergences under certain assumptions.

\vspace{1mm}
In 2011, Macaj and Sleziak \cite{MS11} introduced the $\mathcal{I}^\mathcal{K}$-convergence of function in a topological space, which was derived from $\mathcal{I}^*$-convergence \cite{LD05} by simply replacing $Fin$ by an arbitrary ideal $\mathcal{K}$. \linebreak Interestingly, $\mathcal{I}^\mathcal{K}$-convergence arose as an independent mode of convergence. Comparisions of \linebreak $\mathcal{I}^\mathcal{K}$-convergence with $\mathcal{I}$-convergence \cite{KSW01} can be found in \cite{BP18, MS11, DDGB19}. A few articles for example \cite{DSS19, DST14} contributed to the study of $\mathcal{I}^\mathcal{K}$-convergence of sequence of functions. Some of the definitions and results of \cite{LD05, MS11} that are used in subsequent sections are listed below. Here $X$ is a topological space and $S$ is a set where ideals are defined.
%
We say that a function $f:S \rightarrow X$ is $\mathcal{I}^{\mathcal{M}}$-convergent to a point $x \in X$ if $\exists M \in \mathcal{I}^*$ such that the function $g:S \rightarrow X$ given by 
\begin{center}
	$g(s)= \begin{cases}
	f(s) ,& \mbox{$s \in M$}\\
	x ,& \mbox{$s \notin M$}	
	\end{cases}$
\end{center}
is $\mathcal{M}$-convergent to $x$, where $\mathcal{M}$ is a convergence mode via ideal.

\vspace{1mm}
If $\mathcal{M} = \mathcal{K}^*$, then $f:S \rightarrow X$ is said to be $\mathcal{I}^{\mathcal{K}^*}$-convergent \cite{MS11} to a point $x \in X$. Also, if $\mathcal{M} = \mathcal{K}$, then $f:S \rightarrow X$ is said to be $\mathcal{I}^{\mathcal{K}}$-convergent \cite{MS11} to a point $x \in X$.
In particular, if $X$ is a discrete space, our immediate observation is that only the $\mathcal{I}$-constant functions are $\mathcal{I}$-convergent, for a given ideal $\mathcal{I}$, $f:S \rightarrow X$ is an $\mathcal{I}$-constant function if it attains a constant value except for a set in $\mathcal{I}$. It follows that $\mathcal{I}$ and $\mathcal{I}^*$ convergence coincide for $X$. Thus, $\mathcal{I}^\mathcal{K}$ and $\mathcal{I}^{\mathcal{K}^*}$-convergence modes also coincide on discrete spaces.

\begin{lemma}{\rm{\cite[Lemma 2.1]{MS11}}}
	If $\mathcal{I}$ and $\mathcal{K}$ are two ideals on a set $S$ and $f: S \rightarrow X$ is a function such that $\mathcal{K}-\lim f = x$, then $\mathcal{I}^\mathcal{K}-\lim f = x$.
\end{lemma}
%
\begin{remark}
	We say two ideals $\mathcal{I}$ and $ {\mathcal{K}}$ satisfy ideality condition if $\mathcal{I} \cup \mathcal{K}$ is an proper ideal \cite{HG09}. Again, $\mathcal{I}$ and $ {\mathcal{K}}$ satisfy ideality condition if and only if $S \neq I \cup K$, for all $I \in \mathcal{I}$, $K \in \mathcal{K}$. 
\end{remark}

The main results of this article are divided into 3 sections. Section 2 is devoted to a comparative study of different convergence modes for example $\mathcal{I}^\mathcal{K}$, $\mathcal{I}^{\mathcal{K}^*}$, $\mathcal{I}$, $\mathcal{K}$, $\mathcal{I} \cup \mathcal{K}$, $(\mathcal{I} \cup \mathcal{K})^*$ etc. We justify the existence of an ideal $\mathcal{J}$, such that the behavior of $\mathcal{I}^{\mathcal{K}}$ and $\mathcal{J}$-convergence coincides in Hausdorff spaces. Then in section 3, we introduce $\mathcal{I}^\mathcal{K}$-sequential space and study its properties. In Section 4 we basically define $\mathcal{I}^{\mathcal{K}}$-cluster point and $\mathcal{I}^{\mathcal{K}}$-limit point of a function in a topological space.  Here we observe that the ideality condition of $\mathcal{I}$ and $\mathcal{K}$ in $\mathcal{I}^\mathcal{K}$-convergence allows to get some effective conclusions. Moreover, we characterize the set of $\mathcal{I}^\mathcal{K}$-cluster points of a function as closed sets.
\vspace{1mm}

Throughout this paper we focus on the proper ideals \cite{HG09} containing $Fin$ ($S \notin \mathcal{I}$).

\section{$\mathcal{I}^\mathcal{K}$-convergence and several comparisons}
In this section, we study some more relations among different convergence modes $\mathcal{I}^\mathcal{K}$, $\mathcal{I}^{\mathcal{K}^*}$, $\mathcal{I} \cup \mathcal{K}$, $(\mathcal{I} \cup \mathcal{K})^*$ etc. We mainly focus on $\mathcal{I}^\mathcal{K}$-convergence where $\mathcal{I} \cup \mathcal{K}$ forms an ideal.  
\begin{proposition}\label{Lp}
	Let $X$ be a topological space and $f : S \rightarrow X$ be a function. Let $\mathcal{I}, \mathcal{K}$ be two ideals on $S$ such that $\mathcal{I} \cup \mathcal{K}$ is an ideal. Then
	\begin{itemize}
		\item[\textup{(i)}]	$ f(s) \rightarrow_{\mathcal{I}^{\mathcal{K}^*}} x \equiv f(s) \rightarrow_{(\mathcal{I} \cup \mathcal{K})^*} x$.
		\item[\textup{(ii)}] $f(s) \rightarrow_{\mathcal{I}^{\mathcal{K}}} x$ implies $f(s) \rightarrow_{\mathcal{I} \cup \mathcal{K}} x$.
	\end{itemize}
\end{proposition}
\begin{proof}
	\begin{itemize}
		\item[\textup{(i)}] Let $f: S \rightarrow X$ be $\mathcal{I}^{\mathcal{K}^*}$-convergent. So, there exists a set $M \in \mathcal{I}^* $ for which the function $g : S \rightarrow X$ such that
		\begin{align*}
		g(s)= 
		\begin{cases}
		f(s) ,& s \in M\\
		x ,& s \notin M	
		\end{cases}
		\end{align*}
		is $\mathcal{K}^*$-convergent to $x$. So, there further exists a set $ N \in \mathcal{K}^*$ for which we can consider the function $h: S \rightarrow X$ such that
		
		\begin{align*}
		h(s) =  
		\begin{cases}
		f(s),& s \in M, ~ s \in N\\
		x,& s \notin M ~or~ s \notin N	
		\end{cases}
		\end{align*}		
		is $Fin$-convergent to $x$. Now, Let $K =N^\complement \in \mathcal{K}$, $I= M^\complement \in \mathcal{I}$ (say). Then 
		\begin{align*}
		h(s) = \begin{cases}
		f(s) ,& ~~~~s \in (I \cup K)^\complement      \\
		x ,&~~~~ s \notin (I \cup K)^\complement.    	
		\end{cases}
		\end{align*}
		In essence, we can conclude $f$ is $(\mathcal{I} \cup \mathcal{K})^*$-convergent to $x$.\\
		Conversely, the function $f: S \rightarrow X$ is $(\mathcal{I} \cup \mathcal{K})^*$-convergent to $x$. So, there exists a set $ P = (I \cup K)^\complement \in (\mathcal{I} \cup \mathcal{K})^*$ for which the function $h:S \rightarrow X$ such that
		\begin{align*}
		h(s)=
		\begin{cases}
		f(s) ,& s \in P\\
		x ,& s \notin P	
		\end{cases}
		\end{align*}
		\begin{align*}
		h(s)=
		\begin{cases}
		f(s) ,& s \in (I \cup K)^\complement\\
		x ,& s \notin (I \cup K)^\complement	
		\end{cases}
		\end{align*}
		is $Fin$-convergent to $x$. Lets consider the function $g : S \rightarrow X$ defined as
		\begin{align*}
		g(s) = 
		\begin{cases}
		f(s) ,& s \in I^\complement\\
		x ,& s \notin I^\complement	
		\end{cases}
		\end{align*}
		for which the function $h:S \rightarrow X$ such that
		\begin{align*}
		h(s)=
		\begin{cases}
		f(s) ,& s \in I^\complement,~s \in  K^\complement\\
		x ,& s \notin (I \cup K)^\complement	
		\end{cases}
		\end{align*}
		is $Fin$-convergent to $x$. 	
		Consequently $f$ is $\mathcal{I}^{\mathcal{K}^*}$-convergent to $x$.
		\item[\textup{(ii)}]
		Let $f:S \rightarrow X$ be $\mathcal{I}^{\mathcal{K}}$-convergent to $x$. So, there exists a set $ M \in \mathcal{I}^*$ for which the function $g:S \rightarrow X$ such that
		\begin{align*}
		g(s)= 
		\begin{cases}
		f(s) ,& s \in M\\
		x ,& s \notin M	
		\end{cases} 
		\end{align*}
		is $\mathcal{K}$-convergent to $x$. Then for each $\mathcal{U}_x$, neighborhood of $x$, we have $\{s : g(s) \notin \mathcal{U}_x \} \in \mathcal{K}$. Accordingly, the set given by $\{s : f(s) \notin \mathcal{U}_x, s \in M \} \in \mathcal{K}$.
		Further $\{s : f(s) \notin \mathcal{U}_x\} \subseteq \{ s : f(s) \notin \mathcal{U}_x, s \in M \} \cup \{s : s \notin M \}$.
		Hence, $\{ s : f(s) \notin \mathcal{U}_x\} \in \mathcal{I} \cup \mathcal{K}$.
	\end{itemize}
\end{proof}	
Following are immediate corollaries of the above proposition provided $\mathcal{I} \cup \mathcal{K}$ is an ideal.
\begin{corollary}
	$\mathcal{I}^{\mathcal{K}^*}$-convergence implies $\mathcal{I}-convergence$.
\end{corollary}
\begin{corollary}
	$\mathcal{I}^{\mathcal{K}^*}$-convergence implies $\mathcal{K}-convergence$.
\end{corollary}

\noindent
Following results in \cite{BP18} are corollaries of the above proposition.
\begin{corollary}
	$\mathcal{I}^{\mathcal{K}}$-convergence implies $\mathcal{I}-convergence$ provided $\mathcal{K}\subseteq \mathcal{I}$.
\end{corollary}

\begin{corollary}
	$\mathcal{I}^{\mathcal{K}}$-convergence implies $\mathcal{K}-convergence$ provided $\mathcal{I}\subseteq \mathcal{K}$.
\end{corollary}
%

Following diagram shows the connections between different convergence modes.
\begin{center}
	$ \mathcal{I} \cup \mathcal{K} \leftarrow \mathcal{I}^{\mathcal{K}} \leftarrow \mathcal{I}^* \rightarrow (\mathcal{I} \cup \mathcal{K})^* \equiv \mathcal{I}^{\mathcal{K}^*} \rightarrow \mathcal{I}^{\mathcal{K}^{\mathcal{J}}} $
\end{center}

\smallskip \noindent

In this segment we are interested to find whether there exists an ideal $\mathcal{J}$ such that the behavior of $\mathcal{I}^{\mathcal{K}}$ and $\mathcal{J}$-convergence coincides. Recalling that a filter-base is a non empty collection closed under finite intersection, we have the following result for a given function $f$ in $X$ by taking an ideal-base to be complement of a filter-base.
\begin{lemma} \label{JIK}
	Let $\mathcal{I}$ and $\mathcal{K}$ be two ideals on $S$ satisfying ideality condition. $f: S \rightarrow X$ be a function on a topological space $X$. If $\mathcal{J}=$ ideal generated by $(\mathcal{K} \cup J)$, for any $J \in \mathcal{I}$. Then $f$ is $\mathcal{J}$-convergent to $x \implies f$ is $\mathcal{I}^\mathcal{K}$-convergent to $x$.
\end{lemma}

\begin{proof}
	Let $f$ be $\mathcal{J}$-convergent, where $\mathcal{J}$= ideal generated by the ideal base $(\mathcal{K} \cup J)$, for any $ J \in \mathcal{I}$.
	Now for $J= M^{c}$, consider the function $ g: S \rightarrow X$ defined as
	\begin{center}
		$g(s)=
		\begin{cases}
		f(s) ,& \mbox{$s \in M$}\\
		x ,& \mbox{$s \notin M$}.	
		\end{cases}$
	\end{center}
	\noindent
	Then,
	\begin{center}
		$\{s \in S:g(s) \notin V\} = \{s \in S: f(s) \notin V, s \in M\} \subseteq \{s \in S: f(s) \notin V\} \setminus \{s \in S: s \notin M\}$.
	\end{center}
	Again, $ \{s \in S: f(s) \notin V\} \setminus J$
$\subseteq (K \cup J) \setminus J$
$\in \mathcal{K}$.
	Subsequently, $g$ is $\mathcal{K}-convergent$ to $x$. Hence, $f$ is $\mathcal{I}^\mathcal{K}$-convergent to $x$.
\end{proof}

\begin{theorem}{\rm{\cite[Theorem 3.1]{BP18}}}
	In a Hausdorff space $X$, each function $f : S \rightarrow X$ possess a unique $\mathcal{I}^\mathcal{K}$-limit provided $\mathcal{I}\cup \mathcal{K}$ is an ideal.
\end{theorem}\noindent

\begin{theorem} \label{IKJ}
	Let $X$ be a Hausdorff Space. Let $f: S \rightarrow X$ be $\mathcal{I}^\mathcal{K}$-convergent to $x$. Then $\exists$ an ideal $\mathcal{J}$ such that $x \in X$ is an $\mathcal{I}^\mathcal{K}$-limit of the function $f$ if and only if $x$ is also a $\mathcal{J}$-limit of $f$ provided $\mathcal{I} \cup \mathcal{K}$ is an ideal.
\end{theorem}
\begin{proof}
	Let $f:S \rightarrow X$ is $\mathcal{I}^{\mathcal{K}}$-convergent to $x$. So, there exists a set $M \in \mathcal{I}^*$ such that $g:S \rightarrow X$ with
	\begin{align*}
		g(s)= 
		\begin{cases}
		f(s) ,& s \in M\\
		x ,& s \notin M	
		\end{cases}
	\end{align*}
	is $\mathcal{K}$-convergent to $x$. Consequently, for each neigfhborhood $\mathcal{U}_x$ of $x$. We have
	\begin{align*}
		\{s \in S: g(s) \notin \mathcal{U}_x \} \in \mathcal{K}.
	\end{align*}
	\begin{align*}
		  \implies \{s \in S : f(s) \notin \mathcal{U}_x, s \in M \} \in \mathcal{K}.
	\end{align*}
	Now, let $J= M^{c}$ and $(\mathcal{K} \cup J)$ is an ideal base provided $(\mathcal{I}\cup \mathcal{K})$ is an ideal. Now we consider $\mathcal{J}$, the ideal generated by $(\mathcal{K} \cup J)$. Then 
	\begin{align*}
		\{s \in S: f(s) \notin \mathcal{U}_x\} \subseteq \{ s \in S : f(s) \notin \mathcal{U}_x, s \in M \} \cup \{s \in S: s \notin M \}.
	\end{align*}
	Therefore, $\{ s \in S : f(s) \notin \mathcal{U}_x\} \in (\mathcal{K} \cup J)$.\\
Converse part of the proof is immediate by lemma \ref{JIK}.
\end{proof}

The following arrow diagram exhibit the equivalence shown in theorem \ref{IKJ}.
\begin{center}
	$ \mathcal{K} \xrightarrow{for\thinspace any\thinspace  J \in \mathcal{I}} \mathcal{J} \rightarrow \mathcal{I}^{\mathcal{K}} \xrightarrow{fixed \thinspace J \in \mathcal{I}} \mathcal{J} \rightarrow \mathcal{I} \cup {\mathcal{K}}$
\end{center}
Comprehensively, we may ask the following question.
\begin{question}
	Whether there exists an ideal $\mathcal{J}$ for $\mathcal{I}^\mathcal{K}$-convergence in a given non-Hausdorff topological space $X$ such that $\mathcal{I}^\mathcal{K} \equiv \mathcal{J}$-convergence?
\end{question}
\section{$\mathcal{I}^\mathcal{K}$-sequential space}
Recently, $\mathcal{I}$-sequential space were defined by S.K. Pal \cite{P14} for an ideal $\mathcal{I}$ on $\omega$. An equivalent definition was suggested by Zhou et al. \cite{ZLS19} and further obtain that class of $\mathcal{I}$-sequential spaces includes sequential spaces \cite{HNV04}. 

\vspace{1mm}
First, recall the notion of $\mathcal{I}$-sequential spaces. Let $X$ be a topological space and $O\subseteq X$ is $\mathcal{I}$-open if no sequence in $X \setminus O$ has an $\mathcal{I}$-limit in $O$. Equivalently, for each sequence $\{x_n: n \in \omega\} \subseteq X\setminus O$ with $x_n \rightarrow_{\mathcal{I}} x \in X$, then $x \in X \setminus O$. Now $X$ is said to be an $\mathcal{I}$-sequential space if and only if each $\mathcal{I}$-open subset of $X$ is open. 

\vspace{1mm}
Here we introduce a topological space namely $\mathcal{I}^\mathcal{K}$-sequential space for given ideals $\mathcal{I}$ and $\mathcal{K}$ on $\omega$.
\begin{definition}
	Let $X$ be a topological space and $O, A \subseteq X$. Then
	\begin{itemize}
		\item[(1)]  $O$ is said to be $\mathcal{I}^\mathcal{K}$-open if no sequence in $X \setminus O$ has an $\mathcal{I}^\mathcal{K}$-limit in $O$. Otherwise, for each sequence $\{x_n: n \in \omega\} \subseteq X\setminus O$ with $x_n \rightarrow_{\mathcal{I}^\mathcal{K}} x \in X$, then $x \in X \setminus O$.
		\item[(2)] A subset $F \subseteq X$ is said to be $\mathcal{I}^\mathcal{K}$-closed if $X \setminus A$ is $\mathcal{I}^\mathcal{K}$-open in $X$.
	\end{itemize}	
\end{definition}
\begin{remark} \label{R14}
	The following are obvious for a topological space $X$ and ideals $\mathcal{I}$ and $\mathcal{K}$ on $\omega$.
	\begin{itemize}
		\item[1.] Each open(closed) set of $X$ is $\mathcal{I}^\mathcal{K}$-open(closed).
		\item[2.] If $A$ and $B$ are $\mathcal{I}^\mathcal{K}$-open (closed), then $A \cup B$ is $\mathcal{I}^\mathcal{K}$-open (closed).
		\item[3.] A topological space $X$ is said to be an $\mathcal{I}^\mathcal{K}$-sequential space if and only if each $\mathcal{I}^\mathcal{K}$-open set of $X$ is open.
	\end{itemize}	
\end{remark}

for $\mathcal{I}=\mathcal{K}$, each $\mathcal{I}^\mathcal{K}$-sequential space coincides with a $\mathcal{I}$-sequential space.  
\begin{observation} \label{m}
	Let $\mathcal{M}_1, \mathcal{M}_2$ be two convergence modes in a topological space $X$ such that $\mathcal{M}_1$-convergence implies $\mathcal{M}_2$-convergence. Then $O \subseteq X$ is $\mathcal{M}_2$-open implies that $O$ is $\mathcal{M}_1$-open.
\end{observation}
\begin{proof}
	Let $O$ be not $\mathcal{M}_1$-open in $X$, then $\exists \{x_n\}$ in $(X \setminus O)$ which is $\mathcal{M}_1$-convergent in $X$. So, $\{x_n\}$ is $(X \setminus O)$ is $\mathcal{M}_2$-convergent in $X$ and hence $O$ is not $\mathcal{M}_2$-convergent.
\end{proof}
\begin{corollary} \label{lc}
	Let $\mathcal{M}_1, \mathcal{M}_2$ be two convergence modes in $X$ such that $\mathcal{M}_1$-convergence implies $\mathcal{M}_2$-convergence in $X$. Then $X$ is a $\mathcal{M}_1$-sequential space implies that $X$ is $\mathcal{M}_2$-sequential space.
\end{corollary}

The following is an example of a topological space which is not $\mathcal{I}^\mathcal{K}$-sequential space.
\begin{example}
	Let $S=[a, b]$ be a closed interval with the countable complement topology $\tau_{cc}$, where $a$, $b \in \mathbb{R}$. Let $A$ be any subset of $S$ and $x_n$ be a sequence in $A$, $\mathcal{I}^\mathcal{K}$-convergent to $x$, provided  $\mathcal{I}, \mathcal{K}$ and $\mathcal{I} \cup \mathcal{K}$ is an ideal i.e, $x_n \rightarrow_{\mathcal{I} \cup\mathcal{K}} x$. Consider the neighborhood $U$ of $x$, be the complement of the set $\{x_n: x_n \neq x\}$  in $S$. Then $x_n =x$ for all $n$ except for a set in the ideal $\mathcal{I} \cup \mathcal{K}$. Therefore, a sequence in any set $A$ can only $\mathcal{I}\cup \mathcal{K}$-convergent to an element of $A$ i.e $C$ is $\mathcal{I}\cup \mathcal{K}$-open. Thus every subset of $C$ is $\mathcal{I}^\mathcal{K}$-sequentially open. But not every subset of $S$ is open. Hence $([a, b]$, $\tau_{cc})$ is not $\mathcal{I}^\mathcal{K}$-sequential.
\end{example}

\begin{proposition}
	Let $X$ be a topological space and $\mathcal{I}_1, \mathcal{I}_2, \mathcal{K}_1, \mathcal{K}_2$ be ideals on $S$. Then the following implications hold:
	\begin{itemize}
		\item[(1)] For $\mathcal{K}_1 \subseteq \mathcal{K}_2$ whenever $U \subseteq X$ is $\mathcal{I}^{\mathcal{K}_2}$-open, then it is $\mathcal{I}^{\mathcal{K}_1}$-open. 
		\item[(2)] For $\mathcal{I}_1 \subseteq \mathcal{I}_2$ whenever $U \subseteq X$ is $\mathcal{I}_2^\mathcal{K}$-open, then it is $\mathcal{I}_1^\mathcal{K}$-open.
	\end{itemize}
\end{proposition}
\begin{proof}
	Let $f: S \rightarrow X$ be a function. Then by Proposition 3.6 in \cite{MS11},
	\begin{center}
		$\mathcal{I}_1^\mathcal{K}-\lim f = x  \implies \mathcal{I}_2^\mathcal{K}-\lim f = x$.\\
		$\mathcal{I}^{\mathcal{K}_1}-\lim f = x \implies \mathcal{I}^{\mathcal{K}_2} -\lim f = x$.
	\end{center}
	By lemma \ref{m} we have the required results correspondingly. 
\end{proof}
\begin{corollary}
	For $X$ be topological space and $ \mathcal{I}_1, \mathcal{I}_2, \mathcal{K}_1, \mathcal{K}_2$ be ideals on $\omega$ where $\mathcal{I}_1 \subseteq \mathcal{I}_2$, $\mathcal{K}_1 \subseteq \mathcal{K}_2$. Then the following observations are valid:
	\begin{itemize}
		\item[(1)] If $X$ is $\mathcal{I}_1^\mathcal{K}$-sequential, then it is $\mathcal{I}_2^\mathcal{K}$-sequential.
		\item[(2)] If $X$ is $\mathcal{I}^{\mathcal{K}_1}$-sequential, then it is $\mathcal{I}^{\mathcal{K}_2}$-sequential.
	\end{itemize}
\end{corollary}

\begin{theorem}\label{TOIK}
	In a topological space $X$, if $O$ is open then $O$ is $\mathcal{I}^\mathcal{K}$-open.
\end{theorem}
\begin{proof}
	Let $O$ be open and $\{x_n\}$ be a sequence in $X \setminus O$. Let $y \in O$. Then there is a neighborhood $U$ of $y$ which contained in $O$. Hence $U$ can not contain any term of $\{x_n\}$. So $y$ is not an $\mathcal{I}^{\mathcal{K}}$-limit of the sequence and $O$ is $\mathcal{I}^{\mathcal{K}}$-open.
\end{proof}

\begin{theorem}
	In a metric space X, the notions of open and $\mathcal{I}^\mathcal{K}$-open coincide.
\end{theorem}
\begin{proof}
	Forward implication is obvious from Theorem \ref{TOIK}.\\
	Conversely, Let $O$ be not open i.e., $\exists y \in O$ such that for all neighborhood of $y$ intersect $(X \setminus O)$. Let $x_n \in (X \setminus O) \cap B(y, \frac{1}{n+1})$. Then $x_n \rightarrow y$. Hence $x_n$ is $\mathcal{I}^\mathcal{K}$-convergent to $y$. Thus $O$ is not $\mathcal{I}^\mathcal{K}$-open.
\end{proof}

\begin{theorem}
	Every first countable space is $\mathcal{I}^\mathcal{K}$-sequential space. 
\end{theorem}
\begin{proof}
	We need to prove the reverse implication.\\
	If $A \subset X$ be not open. Then $\exists y \in A$ such that every neighborhood of $y$ intersects $X \setminus A$. Let $ \{U_n : n \in \omega \} $ be a decreasing countable basis at $y$ (say). Consider $x_n \in (X \setminus A) \cap U_n$. Then for each neighborhood $V$ of $y$, $\exists n \in \omega$ with $U_n \subset V$. So, $x_m \in V, \forall m \geq n$ i.e $x_n \rightarrow y$. Hence $x_n \rightarrow_\mathcal{K} y$. Therefore, $A$ is not $\mathcal{I}^\mathcal{K}$-open.
\end{proof}

The following theorem about continuous mapping was also proved by Banerjee et al. \cite{BP18}. However, we have given here an alternative approach to prove.
\begin{theorem}\label{IKc}
	Every continuous function preserves  $\mathcal{I}^\mathcal{K}$-convergence.
\end{theorem}
\begin{proof}
	Let $X$ and $Y$ be two topological spaces and  $c:X \rightarrow Y$ be a continuous function. Let $f:S \rightarrow X$ be $\mathcal{I}^{\mathcal{K}}$-convergent. So $\exists M \in \mathcal{I}^*$ such that $g: S \rightarrow X$ given by
	\begin{center}
		$g(s)$= 
		$\begin{cases}
		f(s) ,& \mbox{$s \in M$}\\
		x ,& \mbox{$s \notin M$}	
		\end{cases}$
	\end{center}
	is $\mathcal{K}$-convergent to $x$. Now the function $c \circ f : S \rightarrow Y$, the image function on $Y$ 
	is $\mathcal{K}$-convergent to $x$ by Theorem 3 in \cite{LD05}. Hence $c \circ f$ is $\mathcal{I}^{\mathcal{K}}$-convergent.
\end{proof}

\smallskip
We now recall the definition of a quotient space. Let $(X, \sim)$ be a topological space with an equivalence relation $\sim$ on $X$. Consider the projection mapping $\prod : X \rightarrow X \slash \sim$ (the set of equivalence classes) and taking $A \subset X \slash \sim$ to be open if and only if $\prod ^{-1} (A)$ is open in $X$, we have the quotient space $X \slash \sim$ induced by $\sim$ on $X$. 
\begin{theorem}\label{TQ}
	Every quotient space of an $\mathcal{I}^\mathcal{K}$-sequential space $X$ is $\mathcal{I}^\mathcal{K}$-sequential. 
\end{theorem}
\begin{proof}
	Let $A \subset X \slash \sim$ be not open. Let $X/\sim$ is a quotient space with an equivalence relation $\sim$, $\prod ^{-1}(A)$ is not open in $X$ i.e., $\exists$ a sequence $\{x_n\}$ in $X \setminus \prod ^{-1}(A)$ which is $\mathcal{I}^{\mathcal{K}}$-convergent to $y \in \prod ^{-1}(A)$. Also $\prod$ is continuous, hence preserves $\mathcal{I}^{\mathcal{K}}$-convergence by Theorem \ref{IKc}. Therefore, $\prod (x_n) \in (X \slash \sim) \setminus A$ with $\mathcal{I}^{\mathcal{K}}$-limit $\prod(y) \in A$. So, $A$ is not $\mathcal{I}^{\mathcal{K}}$-open i.e., $X \slash \sim$ is $\mathcal{I}^{\mathcal{K}}$-sequential. 
\end{proof}

\smallskip
\noindent
Following result is immediate via Proposition \ref{lc}.
\begin{theorem}
	Every sequential space is an $\mathcal{I}^\mathcal{K}$-sequential space.
\end{theorem}

Recall that a topological space $X$ is said to be of countable tightness, if for $A \subseteq X$ and  $ x \in \bar{A}$, then $x \in \bar{C}$ for some countable subset $C \subseteq A$. Every sequential and $\mathcal{I}$-sequential space is of countable tightness \cite{ZLS19}.
\begin{proposition}
	Every  $\mathcal{I}^{\mathcal{K}}$-sequential space $X$ is of countable tightness. 
\end{proposition}
\begin{proof}
	Let $X$ be an $\mathcal{I}^{\mathcal{K}}$-sequential space and $A \subseteq X$. Consider $[A]_\omega  = \bigcup \{ \overline{B} : B$ is a countable subset of $A \} $. Clearly, $A \subseteq [A]_\omega \subseteq \overline{A}$. We claim that, $[A]_\omega$ is $\mathcal{I}^{\mathcal{K}}$-closed in $X$. Consider $\{x_n\}$ be a sequence in $[A]_\omega$, $\mathcal{I}^\mathcal{K}$-convergent to $x \in X$. Since $x_n \in [A]_\omega$, then we can find a countable subset $B$ of $A$ such that $x_n \in \overline{B}$ for all $n \in \omega$. Since $X$ be an $\mathcal{I}^{\mathcal{K}}$-sequential space, so $\overline{B}$ is $\mathcal{I}^{\mathcal{K}}$-closed, thus $x \in \overline{B} \subseteq [A]_\omega$, and further $[A]_\omega$ is $\mathcal{I}^{\mathcal{K}}$-closed in $X$.
	
	Now, let $X$ be an $\mathcal{I}^{\mathcal{K}}$-sequential space and $A$ be a subset of $X$. Since the set $[A]_\omega$ is closed in $X$, and $[A]_\omega \subseteq \overline{A} \subseteq \overline{[A]_\omega}$, thus $\overline{A} = [A]_\omega$. If $x \in \overline{A}$, then $x \in [A]_\omega$, and further, there exists a countable subset $C$ of $A$ such that $x \in \overline{C}$, i.e., $X$ is of countable tightness.
\end{proof}
Now we show that every $\mathcal{I}^{\mathcal{K}}$-sequential space is hereditary with respect to $\mathcal{I}^{\mathcal{K}}$-open ($\mathcal{I}^{\mathcal{K}}$-closed) subspaces. First we have the following lemma.

\begin{lemma} {\rm{\cite[Lemma 2.4]{ZLS19}}} \label{Ieqv}
	Let $\mathcal{I}$ be an ideal on $\omega$ and $x_n$, $y_n$ be two sequences in a topological space $X$ such that $\{n \in \omega : x_n \neq y_n \} \in \mathcal{I}$. Then $x_n \rightarrow_\mathcal{I} x$ if and only if $y _n \rightarrow_\mathcal{I} x$.
\end{lemma}

\begin{theorem}
	If $X$ is an $\mathcal{I}^{\mathcal{K}}$-sequential space then every $\mathcal{I}^{\mathcal{K}}$-open ($\mathcal{I}^{\mathcal{K}}$-closed) subspaces of $X$ is $\mathcal{I}^{\mathcal{K}}$-sequential.
\end{theorem}

\begin{proof}
	Let $X$ be an $\mathcal{I}^{\mathcal{K}}$-sequential space. Suppose that $Y$ is an $\mathcal{I}^{\mathcal{K}}$-open subset of $X$. Then $Y$ is also open in $X$. We anticipate $Y$ to be $\mathcal{I}^{\mathcal{K}}$-sequential space.\\
	Consider $U$ to be $\mathcal{I}^{\mathcal{K}}$-open in $Y$. Here $Y$ is open, so we claim that $U$ is open in $X$. Since $X$ is $\mathcal{I}^{\mathcal{K}}$-sequential space, we need to show that $U$ is $\mathcal{I}^{\mathcal{K}}$-open in $X$. Contra-positively, take $U$ be not $\mathcal{I}^{\mathcal{K}}$-open in $X$. Then, $\exists \{x_n\}$ in $X \setminus U$ such that $x_n \rightarrow_{\mathcal{I}^{\mathcal{K}}} x$ $(\in U.)$ i.e. $\exists M \in {\mathcal{I}}^*$ such that $x_{n_k} \rightarrow_{\mathcal{K}} x$, where $n_k \in M$ and $x_{n_k} \in X \setminus U$. Now $\{n_k : x_{n_k} \notin Y\} \in \mathcal{K}$. For a point $y \in Y\setminus U$ (assume), Now Consider a sequence $\{y_n\}$ such that $y_n= x_n$ for $n \in M$ and $y_n= y_{n_k}$ for $n \notin M$ where $\{y_{n_k}\}$ is defined as $y_{n_k}= x_{n_k}$ for $x_{n_k} \in Y$ and $y_{n_k}= y$ for $x_{n_k} \notin Y$. Then by Lemma \ref{Ieqv}, $\{y_{n_k}\}$ is $\mathcal{K}$-convergent to $x$. Hence $\{y_n\}$ is $\mathcal{I}^\mathcal{K}$-convergent to $x$. So $U$ is not $\mathcal{I}^{\mathcal{K}}$-open in $Y$. That is a contradiction to our assumption.
	
	Let $Y$ be an $\mathcal{I}^{\mathcal{K}}$-closed subset of $X$. Then $Y$ is closed in $X$. For any $\mathcal{I}^{\mathcal{K}}$-closed subset $F$ of $Y$, it is sufficient to show that $F$ is closed in $X$.  Since $X$ is an $\mathcal{I}^{\mathcal{K}}$-sequential space, it is enough to show that $F$ is $\mathcal{I}^{\mathcal{K}}$-closed in $X$. Therefore, let $\{x_n: n \in \omega\}$ be an arbitrary sequence in $F$ with $x_n \rightarrow_{\mathcal{I}^ \mathcal{K}} x$ in $X$. We claim that $x \in F$. Indeed, since $Y$ is closed, we have $x \in Y$, and then it is also clear that $x \in F$ since $F$ is an $\mathcal{I}^{\mathcal{K}}$-closed subset of $Y$.
\end{proof}

\begin{proposition}
	The disjoint topological sum of any family of $\mathcal{I}^{\mathcal{K}}$-sequential spaces is $\mathcal{I}^{\mathcal{K}}$-sequential.
\end{proposition}
\begin{proof}
	Let $(X_{\alpha})_{\alpha \in \Delta}$ be a family of $\mathcal{I}^\mathcal{K}$-sequential space and $X= {\oplus}_{\alpha \in \Delta} X_{\alpha}$. We claim that $X$ is $\mathcal{I}^\mathcal{K}$-sequential space. Let $F$ be $\mathcal{I}^\mathcal{K}$-closed in $X$. For each $\alpha \in \Delta$, $X_{\alpha}$ is closed in $X$ i.e., $X_{\alpha}$ is $\mathcal{I}^\mathcal{K}$-closed in $X$. Hence, $F \cap X_{\alpha}$ is $\mathcal{I}^\mathcal{K}$-closed in $X$ by Remark \ref{R14}. As $(F \cap X_{\alpha}) \subseteq X_{\alpha}$ i.e. $F \cap X_{\alpha}$ is closed in $X_{\alpha}$. Now $F$ is closed in $X \equiv X\setminus F$ is open in $X \equiv {\cup}_{\alpha}(X_{\alpha} \setminus F)$ is open in $X$ if and only if $X_{\alpha} \setminus F$ is open in $X_{\alpha} \equiv F \cap X_{\alpha}$ is closed in $X_\alpha$. Hence $F$ is closed in $X$.
\end{proof}

\section{$\mathcal{I}^{\mathcal{K}}$-cluster point and $\mathcal{I}^{\mathcal{K}}$-limit point}
The notions $\mathcal{I}$-cluster point and $\mathcal{I}$-limit point in a topological space $X$ were defined by Das et al. \cite{LD05} and also characterized $C_x(\mathcal{I})$, the collection of all $\mathcal{I}$-cluster points of a given sequence $x= \{x_n\}$ in $X$, as closed subsets of $X$ (Theorem 10, \cite{LD05}). Here we define $\mathcal{I}^\mathcal{K}$-notions of cluster point and limit points for a function in $X$. 

For $\mathcal{I}^*$-convergence, $\mathcal{I} \cup Fin$ is an ideal, thereupon $\mathcal{I}$ and $Fin$ satisfy ideality condition. Moreover we assume ideality condition of $\mathcal{I}$ and $\mathcal{K}$ in $\mathcal{I}^{\mathcal{K}}$-convergence to investigate some results. 

\begin{definition}\label{Lb}
	Let $f: S \rightarrow X$ be a function and $\mathcal{I}$, $\mathcal{K}$ be two ideals on $S$. Then $x \in X$ is called an $\mathcal{I}^{\mathcal{K}}$-cluster point of $f$ if there exists $M \in \mathcal{I}^*$ such that the function $g: S \rightarrow X$ defined by 
	\begin{center}
		g(s)= $\begin{cases}
		f(s) ,& \mbox{$s \in M$}\\
		x ,& \mbox{$s \notin M$}	
		\end{cases}$
	\end{center}
	has a $\mathcal{K}$-cluster point $x$, i.e., $\{ s \in S : g(s) \in U_x \} \notin \mathcal{K}$.
\end{definition}

\begin{definition}\label{La}
	Let $f: S \rightarrow X$ be a function and $\mathcal{I}$, $\mathcal{K}$ be two ideals on $S$. Then $x \in X$ is called an $\mathcal{I}^{\mathcal{K}}$-limit point of $f$ if there exists $M \in \mathcal{I}^*$ such that for the function $g: S \rightarrow X$ defined by 
	\begin{center}
		g(s)= $\begin{cases}
		f(s) ,& \mbox{$s \in M$}\\
		x ,& \mbox{$s \notin M$}	
		\end{cases}$
	\end{center}
	has a ${\mathcal{K}}$-limit point $x$.
\end{definition}
\noindent
For $\mathcal{I} = \mathcal{K}$, we know the convergence modes $\mathcal{I}^\mathcal{K} \equiv \mathcal{I} \equiv \mathcal{K}$. Hence definitions \ref{Lb} and \ref{La} generalizes the definitions of $\mathcal{I}$ or $\mathcal{K}$-(limit point and cluster point) correspondingly. Again, for nets in a topological space $\mathcal{I}$-limit points and $\mathcal{I}$-cluster points coincide \cite{LD07}. Therefore, $\mathcal{I}^{\mathcal{K}}$-cluster points and $\mathcal{I}^{\mathcal{K}}$-limit points of nets also coincide.

Following the notation in \cite{KSW01}, we denote the collection of all $\mathcal{I}^\mathcal{K}$-limit points and $\mathcal{I}^\mathcal{K}$-cluster points of a function $f$ in a topological space $X$ by $L_f(\mathcal{I}^\mathcal{K})$ and  $C_f(\mathcal{I}^\mathcal{K})$ respectively. We observe that $C_f(\mathcal{I}^\mathcal{K}) \subseteq  C_f(\mathcal{K})$ and $L_f(\mathcal{I}^\mathcal{K}) \subseteq L_f(\mathcal{K})$. We also observe that $L_f(\mathcal{I}^*) = L(\mathcal{I}^*)$, where $L(\mathcal{I}^*)$ denote the collection of $\mathcal{I}^*$-limits of $f$.
\begin{lemma}
	If $\mathcal{I}$ and $\mathcal{K}$ be two ideal then $L_f(\mathcal{I}^\mathcal{K}) \subseteq C_f(\mathcal{I}^\mathcal{K})$.
\end{lemma}
\begin{proof}
	Since $L_f(\mathcal{K}) \subseteq C_f(\mathcal{K})$ for an ideal $\mathcal{K}$, hence the result is immediate.
\end{proof}
We have the following lemma provided the ideals $\mathcal{I}$ and $\mathcal{K}$ satisfy ideality condition.
\begin{lemma}\label{lem}	
	$C_f(\mathcal{I} \cup \mathcal{K}) \subseteq C_f(\mathcal{I}^\mathcal{K})$.
\end{lemma}
\begin{proof}
	Let $y$ be not a $\mathcal{I}^\mathcal{K}$-cluster point of $x=\{x_n\}_{n \in \omega}$. Then for all $M \in \mathcal{I}^*$ such that for the function $g : S \rightarrow X$ defined by 
	\begin{center}
		g(s)= $\begin{cases}
		f(s) ,& \mbox{$s \in M$}\\
		x ,& \mbox{$s \notin M$},	
		\end{cases}$
	\end{center}
	the set $\{ s \in S : g(s) \in U_x \} \in \mathcal{K}$. Since $\{s : f(s) \in U_x \} \subseteq \{s : g(s) \in U_x \} \in \mathcal{K}$.\\
	i.e.  $ \{ s : f(s) \in U_x\} \in \mathcal{I} \cup \mathcal{K}$. Hence $y$ is not a $(\mathcal{I} \cup \mathcal{K})$-cluster point of $x$.
\end{proof}
Since above set inequalities signify the implication $ \mathcal{K} \rightarrow \mathcal{I}^{\mathcal{K}} \rightarrow \mathcal{I} \cup \mathcal{K}$, We expect the following conclusion.
\begin{conjecture}
	$L_f(\mathcal{I} \cup \mathcal{K}) \subseteq L_f(\mathcal{I}^\mathcal{K})$.
\end{conjecture}

\smallskip \noindent
For sequential criteria in \cite{KSW01}, we observe the following result.
\begin{theorem}\label{11}	
	Let $\mathcal{I}$, $\mathcal{K}$ be two ideals on $\omega$ and $X$ be a topological space. Then
	\begin{itemize}
		\item[\textup{(i)}]  For $x=\{x_n\}_{n \in \omega}$, a sequence in $X$; $C_x(\mathcal{I}^\mathcal{K})$ is a closed set.
		\item[\textup{(ii)}] If $(X, \tau)$ is closed hereditary separable and there exists a disjoint sequence of sets $\{P_n\}$ such that $P_n \subset \omega$, $P_n \notin \mathcal{I}, \mathcal{K}$ for all $n$, then for every non empty closed subset $F$ of $X$, there exists a sequence $x$ in $X$ such that $F = C_x(\mathcal{I}^\mathcal{K})$ provided $\mathcal{I} \cup \mathcal{K}$ is an ideal.
	\end{itemize}
\end{theorem}
\begin{proof}
	Consider the sequence $x=\{x_n\}$ in $X$ and $\mathcal{I}$, $\mathcal{K}$ be the two ideals on $\omega$.
	\begin{itemize}
		\item[\textup{(i)}] Let $y \in  \overline{C_x(\mathcal{I}^\mathcal{K})}$; the derived set of $C_x(\mathcal{I}^\mathcal{K})$. Let $U$ be an open set containing $y$. It is clear that $U \cap C_x(\mathcal{I}^\mathcal{K}) \neq \phi $. Let $p \in (U \cap C_x(\mathcal{I}^\mathcal{K}))$ i.e., $p \in U$ and $p \in C_x(\mathcal{I}^\mathcal{K})$. Now there exist a set $M \in \mathcal{I}^*$, such that $\{y_n\}_{n \in \omega}$ given by $y_n = x_n$ if $n \in M$ and $p$, otherwise; we have $\{n \in \omega : y_n \in U\} \notin \mathcal{K}$. Consider the sequence   $\{z_n\}_{n \in \omega}$ given by $z_n = x_n$ if $n \in M$ and $y$, otherwise; then $\{n \in \omega : z_n \in U\} = \{n \in \omega : y_n \in U\} \notin \mathcal{K}$. Hence $y \in C_x(\mathcal{I}^\mathcal{K})$.
		\vspace{1mm}
		\item[\textup{(ii)}] Being a closed subset of $X$, $F$ is separable. Let $S=\{s_1, s_2,...\} \subset F$ be a countable set such that $\overline{S}=F$. Consider $x_n = s_i$ for $n \in P_i$. Thus we have the subsequene $\{k_n\}$ of $\{n\}$ for which assume the sequence $x=\{x_{n_k}\}$. Let $y\in C_x(\mathcal{K})$ (taking $y \neq s_i$ otherwise if $y= s_i$ for some $i$, then $y$ is eventually in $F$). We claim $C_x(\mathcal{K}) \subset F$. Let $U$ be any open set containing y. Then $\{n: x_{n_k} \in U\} \notin \mathcal{K}$ and hence non empty i.e., $s_i \in U$ for some $i$. Therefore $F \cap U$ is non empty, So $y$ is a limit point of $F$ and closedness of $F$ gives $y \in F$. Hence $C_x(\mathcal{K}) \subset F$. Further $C_x(\mathcal{I}^\mathcal{K}) \subseteq  C_x(\mathcal{K}) \subset F$.\\
		Conversely, for $a \in F$ and $U$ be an open set containing $a$, then there exists $s_i \in S$ such that $s_i \in U$. Then $\{n: x_{n_k} \in U\} \supset P_i$ ($\notin \mathcal{K}$, $\mathcal{I}$). Thus $\{n: x_{n_k} \in U\} \notin (\mathcal{I}\cup \mathcal{K})$ i.e., $a \in C_x(\mathcal{I} \cup \mathcal{K})$. On the otherhand, by lemma \ref{lem}, $C_f(\mathcal{I} \cup \mathcal{K}) \subseteq C_f(\mathcal{I}^\mathcal{K})$. So we get the reverse implication. 
	\end{itemize}	
\end{proof}
\begin{remark}
	Theorem \ref{11} generalizes Theorem 10 in \cite{LD05}, it follows by letting $\mathcal{I} =\mathcal{K}$ in the above theorem.
\end{remark}

\section*{ Acknowledgement}
	The first author would like to thank the University Grants Comission (UGC) for awarding the junior research fellowship vide UGC-Ref. No.: 1115/(CSIR-UGC NET DEC. 2017), India.

\end{document}